\documentclass[12pt]{amsart}

\usepackage{amssymb}
\usepackage{bm}
\usepackage{graphicx}
\usepackage{psfrag}
\usepackage{color}
\usepackage{url}
\usepackage{algpseudocode}
\usepackage{fancyhdr}
\usepackage{xy}
\input xy
\xyoption{all}

\newtheorem*{maintheorem*}{Main Theorem}
\newtheorem{theorem}{Theorem}[section]
\newtheorem{prop}[theorem]{Proposition}
\newtheorem{lemma}[theorem]{Lemma}
\newtheorem{cor}[theorem]{Corollary}

\theoremstyle{definition}
\newtheorem{definition}[theorem]{Definition}
\newtheorem{remark}[theorem]{Remark}
\newtheorem{example}[theorem]{Example}
\numberwithin{equation}{section}

\newcommand{\nn}{\mathbb{N}}
\newcommand{\qq}{\mathbb{Q}}
\newcommand{\rr}{\mathbb{R}}
\newcommand{\zz}{\mathbb{Z}}

\newcommand\pval{\mathsf{v}_p}

\keywords{Puiseux monoids, transfer homomorphisms, finitely generated monoids, strongly primary monoids, Krull monoids, C-monoids}

\begin{document}
	
	\mbox{}
	\title{Puiseux monoids and transfer homomorphisms}
	\author{Felix Gotti}
	\address{Department of Mathematics\\UC Berkeley\\Berkeley, CA 94720}
	\email{felixgotti@berkeley.edu}
	\date{\today}
	
	\begin{abstract}
		There are several families of atomic monoids whose arithmetical invariants have received a great deal of attention during the last two decades. The factorization theory of finitely generated monoids, strongly primary monoids, Krull monoids, and C-monoids are among the most systematically studied. Puiseux monoids, which are additive submonoids of $\qq_{\ge 0}$ consisting of nonnegative rational numbers, have only been studied recently. In this paper, we provide evidence that this family comprises plenty of monoids with a basically unexplored atomic structure. We do this by showing that the arithmetical invariants of the well-studied atomic monoids mentioned earlier cannot be transferred to most Puiseux monoids via homomorphisms that preserve atomic configurations, i.e., transfer homomorphisms. Specifically, we show that transfer homomorphisms from a non-finitely generated atomic Puiseux monoid to a finitely generated monoid do not exist. We also find a large family of Puiseux monoids that fail to be strongly primary. In addition, we prove that the only nontrivial Puiseux monoid that accepts a transfer homomorphism to a Krull monoid is $\nn_0$. Finally, we classify the Puiseux monoids that happen to be C-monoids.
	\end{abstract}

\maketitle

\section{Introduction} \label{sec:intro}

The study of the phenomenon of non-unique factorizations in the ring of integers $\mathcal{O}_K$ of an algebraic number field $K$ was initiated by L. Carlitz in the 1950's, and it was later carried out on more general integral domains. As a result, many techniques to measure the non-uniqueness of factorizations in several families of integral domains were systematically developed during the second half of the last century (see \cite{dA97} and references therein). However, it was not until recently that questions about the non-uniqueness of factorizations were abstractly formulated in the context of commutative cancellative monoids. This was possible because most of the factorization-related questions inside an integral domain are purely multiplicative in essence. The fundamental goal of abstract (or modern) factorization theory is to measure how far is a commutative cancellative monoid from being factorial by using different arithmetical invariants.

At this point, the arithmetical invariants of several families of atomic monoids have been intensively studied. Finitely generated monoids, strongly primary monoids, Krull monoids, and C-monoids are among the most studied. These families of monoids not only have very diverse arithmetical properties, but also have proved to be useful in the study of the factorization theory of less-understood atomic monoids via transfer homomorphisms. A monoid homomorphism is said to be transfer if somehow it allows to shift the atomic structure of its codomain back to its domain (see Definition~\ref{def:transfer homomorphism}). Therefore if one is willing to know the factorization invariants of a given monoid, it suffices to find a transfer homomorphism from such a monoid to a better-understood monoid and carry over the desired factorization properties.

Puiseux monoids were recently introduced as a rational generalization of numerical monoids. Many families of atomic Puiseux monoids were explored in \cite{fG16,fG17,GG17}, and their elasticity was studied in \cite{GO17}. However, it is still unanswered whether the non-unique factorization behavior in Puiseux monoids is somehow similar to that of some of the monoids whose factorization properties are already well-understood. To give a partial answer to this, we will determine which atomic Puiseux monoids can be the domain of a transfer homomorphism to some of the monoids whose arithmetical invariants have already been studied. In particular, we consider finitely generated monoids, Krull monoids, and C-monoids as our transfer codomains.

The content of this paper is organized as follows. In Section~\ref{sec:background}, we establish the notation we shall be using later, and we formally present most of the fundamental concepts needed in this paper. Then, in Section~\ref{sec:Morphisms between PMs}, we show that homomorphisms between Puiseux monoids can only be given by rational multiplication, which will allow us to characterize the transfer homomorphisms between Puiseux monoids. We also present a family of Puiseux monoids whose members have $\zz$ as their group of automorphisms. Section~\ref{sec:finite transfer PMs} is devoted to characterize the Puiseux monoids admitting a transfer homomorphism to some finitely generated monoid. Then, in Section~\ref{sec:finitary PM} we investigate which Puiseux monoids are strongly primary. Finally, in Section~\ref{sec:PM are not transfer Krull}, we prove that the only Puiseux monoid that is transfer Krull is the additive monoid $\nn_0$. We use this information to classify the Puiseux monoids which happen to be C-monoids.
\medskip

%%%%%%%%%%%%%%%%%%%%%%%%%%%%%
\section{Background} \label{sec:background}

To begin with let us introduce the fundamental concepts related to our exposition as an excuse to establish the notation we need. The reader can consult Grillet \cite{pG01} for information on commutative semigroups and Geroldinger and Halter-Koch \cite{GH06b} for extensive background in non-unique factorization theory of atomic monoids.

Throughout this sequel, we let $\mathbb{N}$ denote the set of positive integers, and we set $\nn_0 := \nn \cup \{0\}$. For $X \subseteq \rr$ and $r \in \rr$, we set $X_{\le r} := \{x \in X \mid x \le r\}$; with a similar spirit we use the symbols $X_{\ge r}$, $X_{< r}$, and $X_{> r}$. If $q \in \qq_{> 0}$, then we call the unique $a,b \in \nn$ such that $q = a/b$ and $\gcd(a,b)=1$ the \emph{numerator} and \emph{denominator} of $q$ and denote them by $\mathsf{n}(q)$ and $\mathsf{d}(q)$, respectively. For each subset $Q$ of $\qq_{>0}$, we call the sets $\mathsf{n}(Q) = \{\mathsf{n}(q) \mid q \in Q\}$ and $\mathsf{d}(Q) = \{\mathsf{d}(q) \mid q \in Q\}$ the \emph{numerator set} and \emph{denominator set} of $Q$, respectively.

As usual, a \emph{semigroup} is a pair $(S, *)$, where $S$ is a set and $*$ is an associative binary operation in $S$; we write $S$ instead of $(S,*)$ provided that $*$ is clear from the context. However, inside the scope of this paper, a \emph{monoid} is a commutative cancellative semigroup with identity (cf. the standard definition of monoid). To comply with established conventions, we will be using simultaneously additive and multiplicative notations; however, the context will always save us from the risk of ambiguity. Let $M$ be a monoid written additively. We set $M^\bullet := M \! \setminus \! \{0\}$ and, as usual, we let $M^\times$ denote the set of units (i.e., invertible elements) of $M$. The monoid $M$ is \emph{reduced} if $M^\times = \{0\}$. For $a, b \in M$, we say that $a$ \emph{divides} $b$ \emph{in} $M$ if there exists $c \in M$ such that $b = a + c$; in this case we write $a \mid_M b$. An element $a \in M \! \setminus \! M^\times$ is an \emph{atom} if whenever $a = u + v$ for some $u,v \in M$, either $u \in M^\times$ or $v \in M^\times$. Atoms are the building blocks in factorization theory; this motivates the especial notation
\[
	\mathcal{A}(M) := \{a \in M \mid a \text{ is an atom of } M\}.
\]
For $S \subseteq M$, we let $\langle S \rangle$ denote the smallest submonoid of $M$ containing $S$, and we say that $S$ \emph{generates} $M$ if $M = \langle S \rangle$. The monoid $M$ is said to be \emph{finitely generated} if it can be generated by a finite set. On the other hand, we say that $M$ is \emph{atomic} if $M = \langle \mathcal{A}(M) \rangle$.
%An element $p \in M$ is \emph{prime} if for all $a, b \in M$ the fact that $p \mid_M ab$ implies that either $p \mid_M a$ or $p \mid_M b$.
A monoid is \emph{factorial} if every element can be written as a sum of primes. As every prime is an atom, every factorial monoid is atomic.

Let $\rho \subseteq M \times M$ be an equivalence relation on $M$, and let $[a]_\rho$ denote the equivalence class of $a \in M$. We say that $\rho$ is a \emph{congruence} if for all $a,b,c \in M$ such that $(a,b) \in \rho$ it follows that $(ca,cb) \in \rho$. Congruences are precisely the equivalence relations that are compatible with the operation of $M$, meaning that $M/\rho := \{[a]_\rho \mid a \in M\}$ is a commutative semigroup with identity (no necessarily cancellative). Two elements $a,b \in M$ are \emph{associates}, and we write $a \simeq b$, if $a = ub$ for some $u \in M^\times$. Being associates defines a congruence relation $\simeq$ on $M$, and $M_{\text{red}} := M/\!\!\simeq$ is called the \emph{associated reduced semigroup} of $M$.
%If $\phi \colon M \to N$ is a monoid homomorphism, then the map $\phi_{\text{red}} \colon M_{\text{red}} \to N_{\text{red}}$ defined by $\phi_{\text{red}}(aM^\times) = \phi(a)N^\times$ is also a monoid homomorphism.

We say that a multiplicative monoid $F$ is \emph{free abelian} with basis $P \subset F$ if every element $a \in F$ can be written uniquely in the form
\[
	a = \prod_{p \in P} p^{\pval(a)},
\]
where $\pval(a) \in \nn_0$ and $\pval(a) > 0$ only for finitely many elements $p \in P$. The monoid $F$ is determined by $P$ up to canonical isomorphism, so we shall also denote $F$ by $\mathcal{F}(P)$. By the fundamental theorem of arithmetic, the multiplicative monoid $\nn$ is free on the set of prime numbers. In this case, we can extend $\pval$ to $\qq_{\ge 0}$ as follows. For $r \in \qq_{> 0}$ let $\pval(r) := \pval(\mathsf{n}(r)) - \pval(\mathsf{d}(r))$ and set $\pval(0) = \infty$.
%The map $\pval \colon \qq_{\ge 0} \to \zz$, called the $p$-\emph{adic valuation} on $\qq_{\ge 0}$, satisfies the following two conditions:
%\begin{eqnarray}
%	&& \pval(rs) = \pval(r) + \pval(s) \ \ \text{for all} \ \ r,s \in \qq_{\ge 0}; \vspace{10pt} \\
%	&& \pval(r + s) \ge \min \{ \pval(r), \pval(s)\} \ \ \text{for all} \ \ r,s \in \qq_{\ge 0}. \vspace{3pt}
%\end{eqnarray}

The free abelian monoid on $\mathcal{A}(M)$, denoted by $\mathsf{Z}(M)$, is called the \emph{factorization monoid} of $M$, and the elements of $\mathsf{Z}(M)$ are called \emph{factorizations}. If $z = a_1 \dots a_n \in \mathsf{Z}(M)$ for some $n \in \nn_0$ and $a_1, \dots, a_n \in \mathcal{A}(M)$, then $n$ is the \emph{length} of the factorization $z$; the length of $z$ is denoted by $|z|$. The unique homomorphism
\[
	\phi \colon \mathsf{Z}(M) \to M \ \ \text{satisfying} \ \ \phi(a) = a \ \ \text{for all} \ \ a \in \mathcal{A}(M)
\]
is called the \emph{factorization homomorphism} of $M$. Additionally, for $x \in M^\bullet$,
\[
	\mathsf{Z}(x) := \phi^{-1}(x) \subseteq \mathsf{Z}(M)
\]
is the \emph{set of factorizations} of $x$. By definition, we set $\mathsf{Z}(0) = \{0\}$. Note that the monoid $M$ is atomic if and only if $\mathsf{Z}(x)$ is not empty for all $x \in M$. For each $x \in M$, the \emph{set of lengths} of $x$ is defined by
\[
	\mathsf{L}(x) := \{|z| : z \in \mathsf{Z}(x)\}.
\]
We say that the monoid $M$ is \emph{half-factorial} if $|\mathsf{L}(x)| = 1$ for all $x \in M$. On the other hand, if $\mathsf{L}(x)$ is a finite set for all $x \in M$, then we say that $M$ is a \emph{BF-monoid}. The \emph{system of sets of lengths} of $M$ is defined by
\[
	\mathcal{L}(M) := \{\mathsf{L}(x) \mid x \in M\}.
\]
The system of sets of lengths is an arithmetical invariant of atomic monoids that has received significant attention in recent years (see \cite{ACHP07,CGLM11} and the literature cited there).

A very special family of atomic monoids is that one comprising all \emph{numerical monoids}, cofinite submonoids of the additive monoid $\nn_0$. We say that a numerical monoid is \emph{proper} if it is strictly contained in $\nn_0$. Each numerical monoid has a unique minimal set of generators, which is finite. Moreover, if $\{a_1, \dots, a_n\}$ is the minimal set of generators for a numerical monoid $N$, then $\mathcal{A}(N) = \{a_1, \dots, a_n\}$ and $\gcd(a_1, \dots, a_n) = 1$. As a result, every numerical monoid is atomic and contains only finitely many atoms. The \emph{Frobenius number} of $N$, denoted by $F(N)$, is the minimum $n \in \nn$ such that $\zz_{> n} \subset N$. An introduction to numerical monoids can be found in \cite{GR09}.

An additive submonoid of $\qq_{\ge 0}$ is called a \emph{Puiseux monoid}. Puiseux monoids are a natural generalization of numerical monoids. However, the general atomic structure of Puiseux monoids drastically differs from that one of numerical monoids. Puiseux monoids are not always atomic; for instance, consider $\langle 1/2^n \mid n \in \nn\rangle$. On the other hand, if an atomic Puiseux monoid $M$ is not isomorphic to a numerical monoid, then $\mathcal{A}(M)$ is infinite. The atomic structure of Puiseux monoids has been studied in \cite{fG17} and \cite{GG17}, where several families of atomic Puiseux monoids were described.
\medskip

%However, a Puiseux monoid is atomic provided $0$ is not a limit point. We say that a Puiseux monoid $M$ is \emph{strongly bounded} if it can be generated by a set of positive rationals $S$ such that $\mathsf{n}(S)$ is bounded.

\section{Homomorphisms Between Puiseux Monoids} \label{sec:Morphisms between PMs}

In this section we present characterizations of homomorphisms and transfer homomorphisms between Puiseux monoids. Let us start by introducing the concept of a transfer homomorphism, which is going to play a central role in this paper.

\begin{definition} \label{def:transfer homomorphism}
	A monoid homomorphism $\theta \colon M \to N$ is said to be a \emph{transfer homomorphism} if the following conditions hold: \vspace{3pt} \\
	\indent (T1) $N = \theta(M) N^\times$ and $\theta^{-1}(N^\times) = M^\times$; \vspace{3pt} \\
	\indent (T2) if $\theta(a) = b_1 b_2$ for $a \in M$ and $b_1,b_2 \in N$, then there exist $a_1, a_2 \in M$ such that $a = a_1 a_2$ and $\theta(a_i) = b_i$ for $i \in \{1,2\}$.
\end{definition}

We proceed to characterize the homomorphisms between Puiseux monoids. This will immediately yield a characterization of those homomorphisms between Puiseux monoids that happen to be transfer homomorphisms.

\begin{prop} \label{prop:homomorphisms between PM}
	If $\phi \colon M \to N$ be a homomorphism between Puiseux monoids, then the following conditions hold.
	\begin{enumerate}
		\item There exists $q \in \qq_{\ge 0}$ such that $\phi(x) = qx$ for all $x \in M$, i.e., $\phi$ is given by rational multiplication.
		\vspace{3pt}
		\item The homomorphism $\phi$ is a transfer homomorphism if and only if it is surjective.
	\end{enumerate}
\end{prop}

\begin{proof}
	Let us argue first that $\phi$ is given by rational multiplication. It is clear that if a map $P \to P'$ between two Puiseux monoids is multiplication by a rational number, then it is a monoid homomorphism. Thus, it suffices to verify that the only homomorphisms of Puiseux monoids are those given by rational multiplication. To do this, consider the Puiseux monoid homomorphism $\varphi \colon P \to P'$. Because the trivial homomorphism is multiplication by $0$, there is no loss in assuming that $P \neq \{0\}$. Let $\{n_1, \dots, n_k\}$ be a minimal set of generators for the additive monoid $N = P \cap \nn_0$. Notice that $N \neq \{0\}$ and, therefore, $k \ge 1$. The fact that $\varphi$ is nontrivial implies that $\varphi(n_j) \neq 0$ for some $j \in \{1,\dots,k\}$. Set $q = \varphi(n_j)/n_j$, and take $r \in P^\bullet$ and $c_1, \dots, c_k \in \nn_0$ satisfying that $\mathsf{n}(r) = c_1 n_1 + \dots + c_k n_k$. Since $n_i \varphi(n_j) = \varphi(n_i n_j) = n_j \varphi(n_i)$ for each $i \in \{1,\dots,k\}$, one obtains
	%\vspace{-4pt}
	\[
		\varphi(r) = \frac 1{\mathsf{d}(r)} \varphi(\mathsf{n}(r)) = \frac 1{\mathsf{d}(r)} \sum_{i=1}^k c_i \varphi(n_i) = \frac 1{\mathsf{d}(r)} \sum_{i=1}^k c_i n_i \frac{\varphi (n_j)}{n_j} = rq.
	\]
	As a result, the homomorphism $\varphi$ is just multiplication by $q \in \mathbb{Q}_{>0}$.
	
	It is easy to see that condition~(2) is a direct consequence of condition~(1), which completes the proof.
\end{proof}

\begin{remark}
	A Puiseux monoid $M$ is said to be \emph{increasing} (resp., \emph{decreasing}) if $M$ can be generated by an increasing (resp., decreasing) sequence of rational numbers. Also, we say that $M$ is \emph{bounded} if $M$ can be generated by a bounded sequence of rational numbers, and it is said to be \emph{strongly bounded} if it can be generated by a sequence of rational numbers whose numerator set is bounded. Finally, $M$ is called \emph{dense} if it contains $0$ as a limit point. Although the definitions just given are not algebraic in nature, we should notice that they are all preserved by Puiseux monoid isomorphisms. This explains why all of them have been useful in the study of the atomic structure of Puiseux monoids (see \cite{fG17} and \cite{GG17}).
\end{remark}

With notation as in Definition~\ref{def:transfer homomorphism}, when $M$ and $N$ are reduced, we can restate the first condition above as \vspace{4pt} \\
\indent (T1') $\theta$ is surjective and $\theta^{-1}(1) = 1$. \vspace{4pt}

We have already mentioned that a transfer homomorphism allows us to shift the atomic structure and the arithmetic of length of factorizations from its codomain to its domain. This property is formally described in the following proposition.

\begin{prop} \cite[Proposition~1.3.2]{GH06a} \label{prop:transfer properties}
	If $\theta \colon M \to N$ is a transfer homomorphism of atomic monoids, then the following conditions hold:
	\begin{enumerate}
		\item $a \in \mathcal{A}(M)$ if and only if $\theta(a) \in \mathcal{A}(N)$; \vspace{3pt}
		\item $M$ is atomic if and only if $N$ is atomic; \vspace{3pt}
		\item $\mathsf{L}_M(x) = \mathsf{L}_N(\theta(x))$ for all $x \in M$; \vspace{3pt}
		\item $\mathcal{L}(M) = \mathcal{L}(N)$, and so $M$ is a BF-monoid if and only if $N$ is a BF-monoid.
	\end{enumerate}
\end{prop}

%The next corollary immediately follows from Proposition~\ref{prop:homomorphisms between PM}.
%
%\begin{cor}
%	A homomorphism between Puiseux monoids is a transfer homomorphism if and only if it is surjective.
%\end{cor}

For a Puiseux monoid $M$, let $\text{Aut}(M)$ denote the group of automorphisms of $M$. As we have seen in Proposition~\ref{prop:homomorphisms between PM}, the set of homomorphisms between Puiseux monoids is very exclusive. In particular, we might wonder whether $\text{Aut}(M)$ is always trivial. However, it is not hard to verify, for instance, that when $M_1 = \langle 1/2^n \mid n \in \nn \rangle$, multiplication by $1/2$ is in $\text{Aut}(M_1)$. This example might not be the most desirable because $M_1$ fails to be atomic; in fact, $M_1$ does not contain any atoms. The next proposition exhibits a family of atomic monoids whose groups of automorphisms are nontrivial. First, let us introduce a family of atomic Puiseux monoids whose atomicity is used in the proof.

For $r \in \qq_{>0}$, the monoid $M_r = \langle r^n \mid n \in \nn \rangle$ is the \emph{multiplicatively $r$-cyclic} Puiseux monoid. If $\mathsf{n}(r), \mathsf{d}(r) > 1$, then $M_r$ is atomic with $\mathcal{A}(M_r) = \{r^n \mid n \in \nn \}$ (see \cite[Theorem~6.2]{GG17}).

\begin{prop}
	Let $r \in \qq_{>0}$ such that $\mathsf{n}(r), \mathsf{d}(r) > 1$. If $M = \langle r^n \mid n \in \zz \rangle$, then $\emph{Aut}(M) \cong \zz$.
\end{prop}

\begin{proof}
	Set $A = \{r^n \mid n \in \zz\}$. For $n \in \zz$, the fact that $r^n A = A$ implies that multiplication by $r^n$ is an endomorphism of $M$ whose inverse is given by multiplication by $r^{-n}$. Thus, multiplication by any integer power of $r$ is an automorphism of $M$. To prove that these are the only elements of $\text{Aut}(M)$, let us first argue that $M$ is atomic with $\mathcal{A}(M) = A$.
	
	Assume first that $r < 1$. Fix $k \in \zz$, and let us check that $r^k \in \mathcal{A}(M)$. To do this notice that the monoid $\langle r^n \mid n \ge k \rangle$ is the isomorphic image (under multiplication by $r^{k-1}$) of the multiplicatively $r$-cyclic Puiseux monoid $M_r$, which is atomic with set of atoms $A = \{r^n \mid n \in \nn\}$. Since $r \in \mathcal{A}(M_r)$, it follows that $r \notin \langle r^n \mid n > 1 \rangle$. Then $r^k \notin \langle r^n \mid n > k \rangle$. As $r < 1$, no atom in $\{r^n \mid n < k\}$ divides $r^k$. Hence $r^k \notin \langle A \setminus \{r^k\} \rangle$ and, therefore, $r^k \in \mathcal{A}(M)$. As a result, $\mathcal{A}(M) = A$.
	
	Now suppose that $r > 1$. As before, fix $k \in \zz$. Because $r > 1$, proving that $r^k \in \mathcal{A}(M)$ amounts to showing that $r^k \notin \langle r^n \mid n < k \rangle$. Let us assume, by way of contradiction, that this is not the case. Then $r^k = a_1 r^{n_1} + \dots + a_t r^{n_t}$ for some $a_1, \dots, a_t \in \nn$ and $n_1, \dots, n_t \in \nn$ with $k > n_1 > \dots > n_t$. As a consequence, $r^{k - n_t + 1} \in \langle r^{n_1 - n_t + 1}, r^{n_2 - n_t + 1}, \dots, r \rangle$, which contradicts the fact that $r^{k - n_t + 1} \in \mathcal{A}(M_r)$. As in the previous case, we conclude that $\mathcal{A}(M) = A$.
	
	By Proposition~\ref{prop:homomorphisms between PM}, any automorphism of $M$ is given by rational multiplication. Take $s \in \qq_{>0}$ such that $\phi_s \in \text{Aut}(M)$, where $\phi_s$ consists in left multiplication by $s$. Because $\phi_s$ must send atoms to atoms, it follows that $sr = \phi_s(r) \in A$. Therefore $s$ must be an integer power of $r$. Hence $\text{Aut}(M)$ is precisely $A$ when seen as a multiplicative subgroup of $\qq$. As $A$ is the infinite cyclic group, the proof follows.
\end{proof}
\medskip

\section{Finite Transfer Puiseux Monoids} \label{sec:finite transfer PMs}

Now we turn to characterize the transfer homomorphisms from Puiseux monoids to finitely generated monoids.

\begin{definition}
	We say that a Puiseux monoid $M$ is \emph{transfer finite} if there exists a transfer homomorphism from $M$ to a finitely generated monoid.
\end{definition}

By the fundamental structure theorem of finitely generated abelian groups, it immediately follows that every finitely generated monoid $F$ is a submonoid of a group $T \times \zz^\beta$ for some finite abelian group $T$ and $\beta \in \nn_0$. In case of $F$ being reduced, it can be thought of as a submonoid of $T \times \nn_0^\beta$.

Condition (T2) in the definition of a transfer homomorphism $\theta \colon M \to F$ is crucial to transfer the factorization behavior of $F$ to $M$. However, the reader might wonder how much the set $\text{Hom}(M,F)$ will increase if we drop condition (T2). Surprisingly, the set of homomorphisms will remain the same as long as we impose $F$ to be reduced. This fact facilitates to classify the Puiseux monoids that happen to be transfer finite, as we will prove in the next theorem. First, notice that if $\phi \colon M \to N$ is a monoid homomorphism, then the map $\phi_{\text{red}} \colon M_{\text{red}} \to N_{\text{red}}$ defined by $\phi_{\text{red}}(aM^\times) = \phi(a)N^\times$ is also a monoid homomorphism.

%Theorem~\ref{thm:morphisms from PM satisfying only (T1)} also allows us to classify the Puiseux moniods that happen to be transfer finite.

\begin{theorem} \label{thm:morphisms from PM satisfying only (T1)}
	Let $M$ be a nontrivial Puiseux monoid, and let $F$ be a finitely generated (additive) monoid.
	
	\begin{enumerate}
		\item If $\theta \colon M \to F$ is a homomorphism satisfying $\theta^{-1}(0) = \{0\}$, then $M$ is isomorphic to a numerical monoid.
		\vspace{3pt}
		\item The Puiseux monoid $M$ is transfer finite if and only if it is isomorphic to a numerical monoid.
		\end{enumerate}
\end{theorem}

\begin{proof}
	We argue first part~(1). It is easy to see that $\theta_{\text{red}} \colon M_{\text{red}} = M \to F_{\text{red}}$ is also a transfer homomorphism. So we can assume, without loss of generality, that $F$ is reduced. Suppose that $F$ is a submonoid of $T \times \nn_0^\beta$, where $T$ is a finite abelian group and $\beta \in \nn_0$. First, assume, by way of contradiction, that $\beta = 0$. In this case, 
	it is not hard to verify that $\theta(M)$ must be a subgroup of $T$. If $\alpha = |\theta(M)|$ and $r \in M^\bullet$, then $\theta(\alpha r) =  \alpha \theta(r) = 0$. This contradicts that $\theta^{-1}(0) = \{0\}$. Thus, $\beta \ge 1$.
	
	Define $\pi \colon T \times \nn_0^\beta \to \nn_0^\beta$ by $\pi(t,\mathbf{v}) = \mathbf{v}$ for all $t \in T$ and $\mathbf{v} \in \nn^\beta$. Let us verify that $\pi(\theta(M))$ is finitely generated. Take $\textbf{x} = (x_1, \dots, x_\beta) \in \pi(\theta(M))^\bullet$, and let $d = \gcd(x_1, \dots, x_\beta)$. We shall verify that $\pi(\theta(M)) \subseteq \langle \textbf{x}/d \rangle$. To do so, consider $\textbf{y} = (y_1, \dots, y_\beta) \in \pi(\theta(M))^\bullet$. Now take $r,s \in M^\bullet$ such that $\pi(\theta(r)) = \textbf{x}$ and $\pi(\theta(s)) = \textbf{y}$, and take $m,n \in \nn$ satisfying that $\gcd(m,n) = 1$ and $mr = ns$. Because
	\[
		m \textbf{x} = \pi(\theta(mr)) = \pi(\theta(ns)) = n \textbf{y},
	\]
	one finds that $mx_i = ny_i$ for $i = 1, \dots, \beta$. As $\gcd(m,n) = 1$, it follows that $n$ divides each $x_i$, i.e., $d/n \in \nn$. As a result,
	\[
		\mathbf{y} = \frac mn \textbf{x} = \bigg(\frac{md}n \bigg) \frac{\mathbf{x}}{d} \in \bigg\langle \frac{\mathbf{x}}{d} \bigg\rangle.
	\]
	Hence $\pi(\theta(M)) \subseteq \langle \textbf{x}/d \rangle$. Because $\langle \mathbf{x}/d \rangle$ is isomorphic to $\nn_0$, it follows that $\pi(\theta(M))$ is finitely generated. 
	
	We show now that $M$ is also finitely generated, which amounts to proving that $\pi \circ \theta \colon M \to \nn_0^\beta$ is injective. First, let us verify that $\pi$ is injective when restricted to $\theta(M)$. As $\theta(M)$ is a submonoid of the reduced monoid $F$, it is also reduced. Suppose that $(t_1, \mathbf{v}), (t_2, \mathbf{v}) \in \theta(M)$, and let us check that $t_1 = t_2$. If $\mathbf{v} = \mathbf{0}$, then $t_1 = t_2 = 0$ because $\theta(M)^\times$ is trivial. Otherwise, there exist $r,s \in M^\bullet$ such that $\theta(r) = (t_1, \mathbf{v})$ and $\theta(s) = (t_2, \mathbf{v})$. Take $m,n \in \nn$ such that $mr = ns$. Since
	\[
		m (t_1, \mathbf{v}) = m \theta(r) = n \theta(s) = n (t_2, \mathbf{v})
	\]
	and $\mathbf{v} \neq \mathbf{0}$, one finds that $m = n$ and, therefore, $r = s$. This, in turn, implies that $t_1 = t_2$. Hence the restriction of $\pi$ to $\theta(M)$ is injective.
	
	To conclude the proof of part~(1), we show that $\theta$ is also injective. Let $r,s \in M$ such that $\theta(r) = \theta(s) \neq 0$. Taking $m,n \in \nn$ satisfying $mr = ns$, we have
	\[
		m \theta(r) = \theta(mr) = \theta(ns) = n \theta(s).
	\]
	Since $\theta(M)$ is reduced, the element $\theta(r)$ must be torsion-free in $T \times \nn_0^\beta$. Thus, $m = n$, which implies that $r = s$. As $\theta^{-1}(0) = \{0\}$, it follows that $|\theta^{-1}(a)| = 1$ for all $a \in \theta(M)$. Therefore $\theta$ is injective, leading us to the injectivity of $\pi \circ \theta$. Now that fact that $\pi(\theta(M))$ is finitely generated implies that $M$ is also finitely generated. Hence $M$ must be isomorphic to a numerical monoid.
	
	Finally, let us argue part (2) of the theorem. For the direct implication, assume that the homomorphism $\theta \colon M \to F$ is a transfer homomorphism. As we did in the proof of part~(1), we can assume that $F$ is a reduced. As both $M$ and $F$ are reduced, condition (T1') yields $\theta^{-1}(0) = \{0\}$. Now it follows by part~(1) that the Puiseux monoid $M$ is isomorphic to a numerical monoid. For the reverse implication, just take $\theta$ to be the identity map.
\end{proof}

Imposing the homomorphism $\theta \colon M \to F$ in Theorem~\ref{thm:morphisms from PM satisfying only (T1)} to satisfy $\theta^{-1}(0) = \{0\}$ is not superfluous even if $M$ is atomic. Then next example sheds some light upon this observation.

\begin{example}
	Let $p_1, p_2, \dots, $ be an enumeration of the odd prime numbers, and let $M = \langle 1/p_n \mid n \in \nn \rangle$. It is not hard to verify that $\mathcal{A}(M) = \{1/p_n \mid n \in \nn\}$. This implies that $M$ is atomic. Now define $\theta \colon M \to \zz_2$ by setting $\theta(0) = 0$, $\theta(r) = 0$ if $\mathsf{n}(r)$ is even, and $\theta(r) = 1$ if $\mathsf{n}(r)$ is odd. It follows immediately that $\theta$ is a surjective monoid homomorphism. However, $M$ is not isomorphic to any numerical monoid because it contains infinitely many atoms.
\end{example}
\medskip

\section{Strongly Primary Puiseux Monoids} \label{sec:finitary PM}

In this section we investigate which Puiseux monoids are strongly primary. In the case of Puiseux monoids, being strongly primary is equivalent to being finitary. In general, finitary monoids provide a common algebraic framework to study not only the arithmetic of strongly primary monoids but also that one of $v$-noetherian $G$-monoids (see \cite[Section~2.7]{GH06b}). The structure of strongly primary monoids was first studied by Satyanarayana in~\cite{mS72} and has received substantial attention in the literature since then (see \cite{GHL07} and references therein). In particular, the class of strongly primary monoids yields multiplicative models for a large class of one-dimensional local domains (see~\cite[Proposition~2.10.7]{GH06b}).

All monoids mentioned in this section are assumed to be reduced. 

\begin{definition}
	Let $M$ be a monoid.
	\begin{enumerate}
		\item A submonoid $S$ of $M$ is called \emph{divisor-closed} provided that for all $x \in M$ and $s \in S$ the fact that $x \mid_M s$ implies that $x \in S$.
		\item The monoid $M$ is called \emph{primary} if it is nontrivial and its only divisor-closed submonoids are $\{0\}$ and $M$.
		\item The monoid $M$ is called \emph{finitary} if $M$ is a BF-monoid and there exist $n \in \nn$ and a finite subset $S \subseteq M^\bullet$ such that $n M^\bullet  \subseteq S + M$.
		\item The monoid $M$ is called \emph{strongly primary} provided that $M$ is both primary and finitary.
	\end{enumerate}
\end{definition}

Let $M$ be a Puiseux monoid, and let $M'$ be a nontrivial proper submonoid of $M$. Notice that for all $x \in M \setminus M'$ and $y \in M'$ satisfying that $x \mid_M y$, the fact that $x \mid_M \mathsf{n}(x) \mathsf{d}(y) y \in M'$ immediately implies that $M'$ is not a divisor-closed submonoid of $M$. Thus, $M$ is primary. On the other hand, suppose that $F$ is a finitely generated monoid, say $F = \langle S \rangle$ for some finite subset $S$ of $F^\bullet$. Then the fact that $F^\bullet = S + F$ immediately implies that $F$ is finitary. In particular, every nontrivial finitely generated Puiseux monoid is finitary and, therefore, strongly primary. The next proposition summarizes the observations made in this paragraph.

\begin{prop} \hfill
	\begin{enumerate}
		\item Every nontrivial Puiseux monoid is primary.
		\vspace{2pt}
		\item Every finitely generated Puiseux monoid is strongly primary.
	\end{enumerate}
\end{prop}

A Puiseux monoid that is not finitely generated may fail to be strongly primary (see, for example, Proposition~\ref{prop:a family of non-finitary PM}). However, in Proposition~\ref{prop:family of strongly primary PM} and Proposition~\ref{prop:multiplicatively cyclic PM are finitary} we exhibit two infinite families of non-finitely generated Puiseux monoids that are strongly primary. To argue Proposition~\ref{prop:family of strongly primary PM}, we will use the following result, which is a weaker version of~\cite[Proposition 4.5]{fG16}.

\begin{prop} \label{prop:sufficient condition for a PM to be BF}
	If $M$ is a Puiseux monoid satisfying that $0$ is not a limit point of~$M^\bullet$, then $M$ is a BF-monoid.
\end{prop}

The next two propositions introduce two families of (non-finitely generated) strongly primary Puiseux monoids.

\begin{prop} \label{prop:family of strongly primary PM}
	Let $p, q \in \nn$ such that $\gcd(p,q) = 1$, and let $\{S_n\}$ be an inclusion-decreasing sequence of numerical monoids. If a function $f \colon \nn \to \nn$ satisfies that $f(1) = 1$ and $q^{f(n+1) - f(n)} - p^n > p \max \{F(S_n), a \mid a \in \mathcal{A}(S_n) \}$ for every $n \in \nn$, then the Puiseux monoid
	\[
		\bigg\langle \frac{q^{f(n)}}{p^n} s \ \bigg{|} \ n \in \nn \ \text{ and} \ s \in S_n \bigg\rangle
	\]
	is strongly primary.
\end{prop}

\begin{proof}
	Set $M = \big\langle q^{f(n)}s/p^n  \mid n \in \nn \ \text{and} \ s \in S_n \big\rangle$, and for each $n \in \nn$ set $A_n = \mathcal{A}(S_n)$. First, we argue that $M$ is a BF-monoid. To do so, observe that for each $n \in \nn$ the fact that $q^{f(n+1) - f(n)} > p \max A_n$ implies that
	\begin{equation} \label{A_n's}
		\min \frac{q^{f(n+1)}}{p^{n+1}} A_{n+1} \ge \frac{q^{f(n+1)}}{p^{n+1}} > \frac{q^{f(n)}}{p^n} \max A_n.
	\end{equation}
	Therefore the generating set $\cup_{n \in \nn} (q^{f(n)}/p^n) A_n$ of $M$ can be listed as an increasing sequence of rational numbers. As $M$ is generated by an increasing sequence of rational numbers, $0$ cannot be a limit point of $M^\bullet$. Hence $M$ is a BF-monoid by Proposition~\ref{prop:sufficient condition for a PM to be BF}.
	
	We proceed to prove that $M$ is finitary. Since $M \cap \nn_0$ is a submonoid of $(\nn_0,+)$, it is atomic and $\mathcal{A}(M \cap \nn_0)$ is finite. Take $S$ to be the finite set $\mathcal{A}(M \cap \nn_0) \cup q A_1 \subset M$. We are done once we verify the inclusion $p M^\bullet \subseteq S + M$. Fix $n \in \nn$ and $a \in A_{n+1}$. Because $q^{f(n)}a \in q^{f(n)}S_{n+1} \subseteq q^{f(n)}S_n  \in M \cap \nn_0$, the element $q^{f(n)}a$ is divisible in $M$ by an element of $S$. On the other hand, $(q^{f(n+1) - f(n)} - p^n)a \ge q^{f(n+1) - f(n)} - p^n > F(S_n)$, which implies that $(q^{f(n+1) - f(n)} - p^n)a \, \frac{q^{f(n)}}{p^n} \in M$. Thus,
	\[
		p\bigg( \frac{q^{f(n+1)}}{p^{n+1}} a \bigg) = q^{f(n)}a + \big(q^{f(n+1) - f(n)} - p^n\big) \frac{q^{f(n)}}{p^n} a \in S + M.
	\]
	In addition, the fact that $qa \in qS_{n+1} \subseteq qS_1$ guarantees $qa$ is divisible in $M$ by some element of $S$. This, in turn, implies that $p \big(q^{f(1)}/p\big)a = qa \in S + M$. As a result, for each $x \in M^\bullet$ it follows that $px \in m(S+M) \subseteq S + M$ for some $m \in \nn$. Hence $pM^\bullet \subseteq S + M$, as desired. 
\end{proof}

\begin{example}
	Consider the Puiseux monoid
	\[
		M = \bigg \langle \frac{3^{n^2 + 1}}{2^n}, \frac{5 \cdot 3^{n^2}}{2^n} \ \bigg{|} \ n \in \nn \bigg \rangle.
	\]
	Taking $S_n$ to be the numerical monoid $\langle 3, 5 \rangle$ for each $n \in \nn$ and defining the function $f \colon \nn \to \nn$ by $f(n) = n^2$, we can rewrite the Puiseux monoid $M$ as follows:
	\[
		M = \bigg\langle \frac{3^{f(n)}}{2^n} s \ \bigg{|} \ n \in \nn \ \text{ and} \ s \in S_n \bigg\rangle.
	\]
	Since $F(S_n) = 7$ for each $n \in \nn$, it follows that
	\[
		3^{f(n+1) - f(n)} - 2^n  = 3^{2n+1} - 2^n > 14 = 2 \max \{3,5, F(S_n)\}.
	\]
	As $f(1) = 1$, Proposition~\ref{prop:family of strongly primary PM} guarantees that $M$ is a strongly primary Puiseux monoid.
\end{example}
\medskip

As in Section~\ref{sec:Morphisms between PMs}, for $r \in \qq_{> 0}$ we let $M_r$ denote the multiplicatively $r$-cyclic Puiseux monoid $\langle r^n \mid n \in \nn \rangle$.

\begin{prop} \label{prop:multiplicatively cyclic PM are finitary}
	For each $r \in \qq_{> 1}$, the Puiseux monoid $M_r$ is strongly primary.
\end{prop}

\begin{proof}
	Since $r > 1$, it follows that $0$ is not a limit point of $M_r^\bullet$. Thus, Proposition~\ref{prop:sufficient condition for a PM to be BF} ensures that $M_r$ is a BF-monoid. On the other hand, it was proved in \cite{GG17} that $\mathcal{A}(M_r) = \{r^n \mid n \in \nn\}$. Now we check that $\mathsf{n}(r)$ divides $\mathsf{n}(r)r^j$ in $M_r$ for every $j \in \nn_0$. If $j=0$, then $\mathsf{n}(r) \mid_{M_r} \mathsf{n}(r)r^j$ follows trivially. Hence it suffices to assume that $j \in \nn$. In this case, the fact that $\mathsf{n}(r)(r-1) = r( \mathsf{n}(r) - \mathsf{d}(r))$ implies that
	\[
		\mathsf{n}(r) r^j - \mathsf{n}(r) = \mathsf{n}(r) (r-1) \sum_{i=0}^{j-1} r^i = r \big(\mathsf{n}(r) - \mathsf{d}(r)\big) \sum_{i=0}^{j-1} r^i = \big(\mathsf{n}(r) - \mathsf{d}(r)\big) \sum_{i=0}^{j-1} r^{i+1} \in M_r.
	\]
	Therefore $\mathsf{n}(r)$ divides $\mathsf{n}(r) r^j$ in $M_r$. To show now that $M_r$ is finitary, we take $n = \mathsf{d}(r)$ and $S = \{\mathsf{n}(r)\}$ and verify that $n M_r^\bullet \subseteq S + M_r$. For $q \in M_r^\bullet$, take $\alpha_1, \dots, \alpha_t \in \nn_0$ such that $q = \sum \alpha_i r^i$, and fix $k \in \{1, \dots, t\}$ such that $\alpha_k > 0$. Because $\mathsf{n}(r)$ divides $\mathsf{n}(r) r^{k-1}$ in $M_r$, there exists $s \in M_r$ satisfying that $\mathsf{n}(r) r^{k-1} = \mathsf{n}(r) + s$. Thus,
	\begin{align*}
		n q &= \mathsf{d}(r) \alpha_k r^k + \sum_{i \neq k} \mathsf{d}(r) \alpha_i r^i \\
			  &= \mathsf{n}(r) r^{k-1} + (\alpha_k - 1) \mathsf{d}(r) r^k + \sum_{i \neq k} \mathsf{d}(r) \alpha_i r^i \\
			  &= \mathsf{n}(r) + \big( s + (\alpha_k - 1) \mathsf{d}(r) r^k + \sum_{i \neq k} \mathsf{d}(r) \alpha_i r^i \big),
	\end{align*}
	which implies that $nq \in S + M_r$. Since $n M_r^\bullet \subseteq S + M_r$, one obtains that $M_r$ is finitary and, therefore, strongly primary.
\end{proof}

We conclude this section providing a family of Puiseux monoids that fail to be strongly primary.

\begin{prop} \label{prop:a family of non-finitary PM}
	Let $\{a_n\}$ be a sequence of positive rational numbers satisfying that $\gcd(\mathsf{d}(a_i), \mathsf{d}(a_j)) = 1$ for any $i \neq j$. Then the Puiseux monoid $\langle a_n \mid n \in \nn \rangle$ is atomic but not strongly primary.
\end{prop}

\begin{proof}
	Set $M = \langle a_n \mid n \in \nn \rangle$. It is not difficult to verify that $\mathcal{A}(M) = \{a_n \mid n \in \nn\}$ from the fact that $\gcd(\mathsf{d}(a_i), \mathsf{d}(a_j)) = 1$ for any $i \neq j$; we leave the details to the reader. This implies, in particular, that $M$ is atomic.
	
	Suppose, by way of contradiction, that $M$ is finitary. Choose $n \in \nn$ and a finite subset $S$ of $M^\bullet$ such that $n M^\bullet \subseteq S + M$. Let $S'$ be a finite subset of $\mathcal{A}(M)$ such that for each $s \in S$ there is at least one atom in $S'$ dividing $s$ in $M$. After substituting $S$ by $S'$, we can assume that $S \subset \mathcal{A}(M)$. Since $n M^\bullet \subseteq S + M$ and $\inf (S+M) \ge \min S$, it follows that $0$ cannot be a limit point of $M^\bullet$. Fix $\epsilon > 0$ such that $\epsilon < \inf M^\bullet$. Since $\gcd(\mathsf{d}(a_i), \mathsf{d}(a_j)) = 1$, there exists $j \in \nn$ such that $\mathsf{d}(a_j) > \max \{n, \max \mathsf{d}(S)\}$. Because $n M^\bullet \subseteq S + M$, one can write
	\begin{equation} \label{eq:finitary}
		n a_j = a_s + \sum_{i=1}^k \alpha_i a_i
	\end{equation}
	for some $k \in \nn$, $a_s \in S$, and $\alpha_i \in \nn_0$. Applying the $\mathsf{d}(a_j)$-valuation to both sides of~(\ref{eq:finitary}) and using the fact that $\gcd(\mathsf{d}(a_i), \mathsf{d}(a_j)) = 1$, it is not hard to find that $\mathsf{d}(a_j)$ divides $n - \alpha_j$. Now $\mathsf{d}(a_j) > n$ yields $n = \alpha_j$. This, along with~(\ref{eq:finitary}), would force $a_s = 0$, which contradicts that $S \subset \mathcal{A}(M)$. Thus, $M$ is not strongly primary.
\end{proof}

\begin{remark}
	The previous proposition not only shows that Puiseux monoids are not, in general, strongly primary, but also illustrates that a natural bounding, ordering, or topological restriction under which a Puiseux monoid is guaranteed to be strongly primary is rather unlikely. For example, consider the Puiseux monoids
	\[
		M_1 = \bigg\langle \frac{1}{p} \ \bigg{|} \ p \ \text{is prime} \bigg\rangle, \
		M_2 = \bigg\langle \frac{p-1}{p} \ \bigg{|} \ p \ \text{is prime} \bigg\rangle, \ \text{and} \
		M_3 = \bigg\langle \frac{p^2+1}{p} \ \bigg{|} \ p \ \text{is prime} \bigg\rangle.
	\]
	It is not hard to check that the sets of atoms of $M_1$, $M_2$, and $M_3$ are precisely the generating sets displayed. Therefore $M_1$ is strongly bounded, $M_2$ is bounded, and $M_3$ is not bounded. We can also see that $M_1$ is a decreasing Puiseux monoid, while $M_2$ is increasing. Furthermore, notice that $0$ is a limit point of $M_1^\bullet$, but $0$ is not a limit point of $M_2^\bullet$. Finally, Proposition~\ref{prop:a family of non-finitary PM} ensures that $M_1$, $M_2$, and $M_3$ are all strongly primary.
\end{remark}
\medskip

%%%%%%%%%%%%%%%%%%%%%%%%%%%%%%%%%%%%%%%%%%%%%%%%%%%%%%%%%
\section{Puiseux Monoids Are Almost Never Transfer Krull} \label{sec:PM are not transfer Krull}

We dedicate this section to show that the atomic structure of Puiseux monoids almost never can be obtained by transferring back that one of Krull monoids; specifically we shall prove that the existence of a transfer homomorphism from a nontrivial Puiseux monoid to a Krull monoid forces the domain to be isomorphic to $(\nn_0,+)$. The we use this information to show that only finitely generated Puiseux monoids admit transfer homomorphisms to C-monoids. Let us start by giving the definition of a Krull monoid. 

\begin{definition} \label{def:Krull monoid}
	A monoid $K$ is called a \emph{Krull monoid} if there is a monoid homomorphism $\varphi \colon K \to D$, where $D$ is a free abelian monoid and $\varphi$ satisfies the following two conditions:
	\begin{enumerate}
		\item if $a, b \in K$ and $\varphi(a) \mid_D \varphi(b)$, then $a \mid_K b$; \vspace{5pt}
		\item for every $d \in D$ there exist $a_1, \dots, a_n \in K$ with $d = \gcd\{\varphi(a_1), \dots, \varphi(a_n)\}$.
	\end{enumerate}
\end{definition}

With notation as in Definition~\ref{def:Krull monoid}, it is easy to see that $K$ is a Krull monoid if and only if $K_{\text{red}}$ is a Krull monoid. The basis elements of $D$ are called the \emph{prime divisors} of $K$. The abelian group Cl$(K) := D/\varphi(K)$ is called the \emph{class group} of $K$ (see \cite[Section~2.3]{GH06b}). As Krull monoids are isomorphic to submonoids of free abelian monoids, Krull monoids are atomic.

%\noindent {\bf Examples of Krull monoids:}
%\begin{enumerate}
%	\item (Algebraic numbe theory) If $K$ is an algebraic number field, and let $\mathcal{O}_K$ be the ring of integers of $K$. Then the multiplicative monoid $\mathcal{O}_K^\bullet$ is a Krull monoid.
%	
%	\item (Hilbert monoids) Take $f \in \zz_{\ge 2}$, the multiplicative monoid $H_f := \{1 + f \nn_0\}$ is a Krull monoid. In particular, the \emph{Hilbert monoid} $H_1$ is an example of Krull monoid. The monoids $H_f$ belong to a larger family of Krull monoids, the \emph{regular congruence monoids} (see \cite[Section~2.11]{GH06b}).
%	
%	\item (Commutative algebra) In general, the multiplicative monoid of every integrally closed Noetherian domain is a Krull monoid. See \cite[Theorem~2.10.2]{GH06b} for a more refined version of the property just stated.
%	
%	\item (Module theory) Let $R$ be a semilocal ring, and let $\mathcal{C}$ be the class of all finitely generated projective $R$-modules. The set $V(\mathcal{C})$ of isomorphism classes of $\mathcal{C}$ is a Krull monoid under the operation $[A] \cdot [B] := [A \oplus B]$ (see \cite{FH00}).
%	
%	If instead $R$ is a commutative local Noetherian ring and $\mathcal{C}$ is the class of all finitely generated $R$-modules, then $V(\mathcal{C})$ is also a Krull monoid under the operation defined in the above paragraph (see \cite{rW01}).
%\end{enumerate}

The factorization theory of Krull monoids has been significantly studied (see \cite{CGP14,GS16} and references therein). The class of Krull monoids contains many well-studied types of monoids, including the multiplicative monoid of the ring of integers of an algebraic number, the Hilbert monoids, and the regular congruence monoids. These and further examples of Krull monoids are presented in \cite[Section~5]{GH06a} and \cite[Section~2.3]{GH06b}.

From the point of view of factorization theory, perhaps the most important family of Krull monoids is that one consisting of block monoids, which we are about to introduce. This is because block monoids capture the essence of the arithmetic of lengths of factorizations in Krull monoids. Let $G$ be an abelian group and $\mathcal{F}(G)$ the free abelian monoid on $G$. An element $X = g_1 \dots g_l \in \mathcal{F}(G)$ is called a \emph{sequence over} $G$. The \emph{length} of $X$ is defined as
\[
	|X| = l =\sum_{g \in G} \mathsf{v}_g(X).
\]
For every $I \subseteq [1, l]$, the sequence $Y = \prod_{i \in I} g_i$ is called a \emph{subsequence} of $X$. The subsequences are precisely the divisors of $X$ in the free abelian monoid $\mathcal{F}(G)$. The submonoid
\[
	\mathcal{B}(G) := \bigg\{ X \in \mathcal{F}(G) \ \bigg{|} \ \sum_{g \in G} \mathsf{v}_g(X) g = 0 \bigg\}
\]
of $\mathcal{F}(G)$ is called the \emph{block monoid} on $G$, and its elements are referred to as \emph{zero-sum sequences} or \emph{blocks over} $G$ (\cite[Section~2.5]{GH06b} is a good general reference on block monoids). Furthermore, if $G_0$ is a subset of $G$, then the submonoid
\[
	\mathcal{B}(G_0) := \{ X \in \mathcal{B}(G) \mid \mathsf{v}_g(X) = 0 \ \emph{ if } \ g \notin G_0 \}
\]
of $\mathcal{B}(G)$ is called the \emph{restriction} of the block monoid $\mathcal{B}(G)$ to $G_0$. For $X \in \mathcal{B}(G_0)$, the \emph{support} of $X$ in $G_0$ is defined to be
\[
	\text{supp}_{G_0}(X) := \{g \in G_0 \mid \mathsf{v}_g(X) > 0\}.
\]
As mentioned before, the relevance of block monoids in the theory of non-unique factorizations lies in the next result.

\begin{prop} \label{prop:relation of Krull monoids and block monoids}\cite[Theorem~3.4.10.3]{GH06b}
	Let $K$ be a Krull monoid with class group $G$ and let $G_0$ be the set of classes of $G$ which contain prime divisors. Then
	\[
		\mathcal{L}(K) = \mathcal{L}(\mathcal{B}(G_0)).
	\]
\end{prop}

As a consequence, understanding the arithmetic of lengths of factorizations in Krull monoids amounts to understanding the same in block monoids.

\begin{definition} \label{def:transfer Krull}
	A Puiseux monoid $M$ is \emph{transfer Krull} if there exist an abelian group $G$, a subset $G_0$ of $G$, and a transfer homomorphism $\theta \colon M \to \mathcal{B}(G_0)$.
\end{definition}

\noindent {\bf Remark:} Our definition of a transfer Krull monoid coincides with the definition given in \cite[Section 4]{aG16}; this is because in the present setting the concepts of a transfer homomorphism and the concept of a weak transfer homomorphism coincide by \cite[Lemma 2.3.(3)]{BS15}.

We denote the field of fractions of an integral domain $R$ by $\mathsf{q}(R)$. For subsets $X,Y$ of $\mathsf{q}(R)$ we set $(X : Y) := \{x \in \mathsf{q}(R) \mid xY \subseteq X\}$. In addition, $R$ is called a \emph{Krull domain} if $R^\bullet$ is a Krull monoid. In this case, the \emph{divisor class group} of $R$, denoted by $\mathcal{C}(R)$, measures the extent to which factorizations in $R$ fail to be unique (see \cite[Section~2.10]{GH06b}). Unlike Krull domains/monoids, which have been central objects in commutative algebra since mid-nineteenth century, transfer Krull monoids (which generalize the concept of Krull monoids) were introduced more recently. Let us proceed to present a few examples of transfer Krull monoids.
\medskip \\

\noindent {\bf Examples of transfer Krull monoids:}
\begin{enumerate}
	
	\item Let $H$ be a half-factorial monoid, and let $\theta \colon H \to \mathcal{B}(\{0\})$ be the map defined by $\theta(h) = 0$ if $h \in \mathcal{A}(H)$ and $\theta(h) = 1$ if $h \in H^\times$. As the map $\theta$ is a transfer homomorphism, it follows that $H$ is a transfer Krull monoid.
	
	\item Let $R$ be a Krull domain, and let $K$ be a subring of $R$ with the same field of fractions. Suppose, in addition, that the following three conditions hold:
	\begin{enumerate}
		\item $R = KR^\times$;
		\item $K \cap R^\times = K^\times$;
		\item $(K : R)$ is a maximal ideal of $K$ (see \cite[Proposition~3.7.5]{GH06b}).
	\end{enumerate} 
	Then the inclusion map $K^\bullet \hookrightarrow R^\bullet$ is a transfer homomorphism and, therefore, $K^\bullet$ is a transfer Krull monoid.
	
	\item Transfer Krull monoids can also be defined in a non-commutative context (see, for instance, \cite{BS15}). Let $R$ be a bounded HNP (hereditary Noetherian prime) ring. If every stably free left $R$-ideal is free, then $R^\bullet$ is a transfer Krull monoid (see \cite[Theorem~4.4]{dS16} for details).
	
%	\item (Module theory) Let $R$ be a semilocal ring, and let $\mathcal{C}$ be the class of all finitely generated projective $R$-modules. The set $V(\mathcal{C})$ of isomorphism classes of $\mathcal{C}$ is a Krull monoid under the operation $[A] \cdot [B] := [A \oplus B]$ (see \cite{FH00}).
%	
%	If instead $R$ is a commutative local Noetherian ring and $\mathcal{C}$ is the class of all finitely generated $R$-modules, then $V(\mathcal{C})$ is also a Krull monoid under the operation defined in the above paragraph (see \cite{rW01}).
\end{enumerate}

There are also many monoids that fail to be transfer Krull. Examples of non-transfer Krull monoids in a non-commutative setting are provided by \cite[Proposition~4.11]{FT17}, \cite[Corollary~4.4]{GS17}, and \cite[Theorem~1.2]{dS13}. On the other hand, Theorem~\ref{thm:Puiseux monoids are not transfer Krull} and the next proposition (which follows from \cite[Theorem~5.5]{GSZ17}) yield examples of non-transfer Krull monoids in a commutative context.

\begin{prop} \label{prop:numerical monoids are not transfer Krull}
	Every proper numerical monoid fails to be transfer Krull.
\end{prop}

The next lemma will be used in the proof of Theorem~\ref{thm:Puiseux monoids are not transfer Krull}.

\begin{lemma} \label{lem:increasingly generated affine semigroups}
	If $\{a_n\}$ is an infinite sequence of positive integers, then there exists $m \in \nn$ such that $a_{m+1} \in \langle a_1, \dots, a_m \rangle$. 
\end{lemma}

\begin{proof}
	If $\{a_n\}$ is bounded there is a term that repeats infinitely many times, making the conclusion of the lemma obvious. Thus, suppose that $\{a_n\}$ is not bounded. Let $\{a_{n_j}\}$ be a subsequence of $\{a_n\}$ satisfying that
	\begin{equation} \label{eq:lemma used in transfer Krull theorem}
		a_{n_{j+1}} > \prod_{i=1}^{j} a_{n_i}
	\end{equation}
	for every $j \in \nn$. Now, for each natural number $j$, set $d_j = \gcd(a_{n_1}, \dots, a_{n_j})$, and notice that $d_{j+1} \mid d_j$ for every $j \in \nn$. Therefore $d_{k+1} = d_k$ must hold for some $k$. In particular, $d_k \mid a_{n_{k+1}}$. On the other hand, condition \eqref{eq:lemma used in transfer Krull theorem} ensures that $a_{n_{k+1}}/d_k$ is greater than the Frobenius number of the numerical monoid $\langle a_{n_1}/d_k, \dots, a_{n_k}/d_k \rangle$. This implies that $a_{n_{k+1}} \in \langle a_{n_1}, \dots, a_{n_k} \rangle$. The lemma follows by taking $m = n_{k+1}-1$.
\end{proof}

Now we are in a position to prove that atomic Puiseux monoids are almost never transfer Krull.

\begin{theorem} \label{thm:Puiseux monoids are not transfer Krull}
	If a nontrivial Puiseux monoid is transfer Krull, then it must be isomorphic to $(\nn_0,+)$.
\end{theorem}

\begin{proof}
	Let $M$ be a nontrivial Puiseux monoid that happens to be transfer Krull. As Krull monoids are atomic, $M$ is atomic by Proposition~\ref{prop:transfer properties}. Let $G$ be an abelian group, and let $\theta \colon M \to \mathcal{B}(G_0)$ be a transfer homomorphism, where $G_0$ is a subset of $G$. Because both $M$ and $\mathcal{B}(G_0)$ are reduced, $\theta^{-1}(\emptyset) = \{0\}$. Assume, by way of contradiction, that $M$ is not isomorphic to a numerical monoid. Take $X \in \mathcal{B}(G_0)^\bullet$ and $r,s \in M^\bullet$ such that $\theta(r) = \theta(s) = X$. Taking $m,n \in \nn$ such that $mr = ns$, one obtains
	\begin{equation} \label{eq:transfer Krull}
		\prod_{g \in G_0} g^{m \mathsf{v}_g(X)} = \theta(r)^m = \theta(s)^n = \prod_{g \in G_0} g^{n \mathsf{v}_g(X)}.
	\end{equation}
	Since $|X| \ge 1$ and $m \mathsf{v}_g(X) = n \mathsf{v}_g(X)$ for every $g \in G_0$, it follows that $m=n$, which yields $r = s$. Hence the preimage under $\theta$ of each element of $\mathcal{B}(G_0)^\bullet$ is a singleton. This, along with the fact that $\theta^{-1}(\emptyset) = \{0\}$, implies that $\theta$ is injective. %As a result, $|\mathcal{A}(M)| = \infty$ forces $|\mathcal{A}(\mathcal{B}(G_0))| = \infty$. Thus, $G$ must be an infinite group.
	In addition, the same equality \eqref{eq:transfer Krull} implies that
	\[
		\text{supp}_{G_0}(\theta(a)) = \text{supp}_{G_0}(\theta(a'))
	\]
	for all $a,a' \in \mathcal{A}(M)$. As a consequence, any two elements of $\theta(M^\bullet)$ have the same support, and we can assume, without loss of generality, that $G_0$ is finite. Let $G_0 =: \{g_1, \dots, g_t\}$ be the common support. List the set $\mathcal{A}(M)$ as a sequence $\{a_n\}$, and let $A_n = \theta(a_n)$ for each $n \in \nn$. Because $\theta$ is injective, $A_i \neq A_j$ when $i \neq j$. Now, for any pair $(i,j) \in \nn^2$, there exist $c_i,c_j \in \nn$ such that $c_i a_i = c_j a_j$. For each $n \in \{1, \dots, t\}$, we can apply $\mathsf{v}_{g_n} \circ \theta$ to the equality $c_i a_i = c_j a_j$ to get $c_i \mathsf{v}_{g_n}(A_i) = c_j \mathsf{v}_{g_n}(A_j)$. After rewriting this equality, one obtains that 
	\begin{equation} \label{eq:main theorem eq2}
		\frac{\mathsf{v}_{g_n}(A_i)}{\mathsf{v}_{g_n}(A_j)} = \frac{c_j}{c_i} = \frac{\mathsf{v}_{g_1}(A_i)}{\mathsf{v}_{g_1}(A_j)}
	\end{equation}
	for each $n \in \{1, \dots, t\}$. On the other hand, notice that Lemma~\ref{lem:increasingly generated affine semigroups} guarantees the existence of $m \in \nn$ and $\alpha_1, \dots, \alpha_m \in \nn_0$ such that
	\begin{equation} \label{eq:main theorem eq3}
		\mathsf{v}_{g_1}(A_{m+1}) = \sum_{i=1}^m \alpha_i \mathsf{v}_{g_1}(A_i).
	\end{equation}
	By \eqref{eq:main theorem eq2}, it follows that the equality \eqref{eq:main theorem eq3} holds when we replace $g_1$ by any other element of $G_0$ (exactly with the same $\alpha_i$'s). As a result, we obtain
	\[
		A_{m+1} = \prod_{j=1}^{|G_0|} g_j^{\mathsf{v}_{g_j}(A_{m+1})} = \prod_{j=1}^{|G_0|} \prod_{i=1}^m g_j^{\alpha_i \mathsf{v}_{g_j}(A_i)} = \prod_{i=1}^m \bigg(\prod_{j=1}^{|G_0|} g_j^{\mathsf{v}_{g_j}(A_i)}\bigg)^{\alpha_i} = \prod_{i=1}^m A_i^{\alpha_i}.
	\]
	This contradicts the fact that $A_{m+1}$ is an atom of the block monoid $\mathcal{B}(G_0)$. Therefore $M$ must be isomorphic to a numerical monoid. Now the direct implication of the proof follows by Proposition~\ref{prop:numerical monoids are not transfer Krull}. For the reverse implication, it suffices to notice that $\theta \colon 1 \mapsto [1_G]$ is an isomorphism from $(\nn_0,+)$ to the block monoid $\mathcal{B}(G)$, where $G$ is the trivial group.
\end{proof}

\begin{cor}
	A Puiseux monoid is a Krull monoid if and only if it can be generated by one element. \\
\end{cor}

Perhaps the second most-systematically studied family of atomic monoids is that one comprising the C-monoids. We would like to know under which conditions a Puiseux monoid happens to be a C-monoid.

Any monoid $M$ can be embedded into a quotient group $\mathsf{g}(M)$, which is unique up to canonical isomorphism. Let $D$ be a multiplicative monoid with quotient group $\mathsf{g}(D)$, and let $M$ be a submonoid of $D$. Two elements $x,y \in D$ are said to be $M$-\emph{equivalent} provided that $x^{-1}M \cap D = y^{-1}M \cap D$. It can be easily checked that being $M$-equivalent defines a congruence relation on $D$. For each $x \in D$, let $[x]^D_M$ denote the congruence class of $x$. The set
\[
	\mathcal{C}^*(M,D) := \big\{[x]^D_M \mid x \in (D \setminus D^\times) \cup \{1\}\big\}
\]
is a commutative semigroup with identity, which is called the \emph{reduced class semigroup} of $M$ in $D$.

\begin{definition} \label{def:C-monoid}
	A monoid $M$ is called a C-\emph{monoid} if it is a submonoid of a factorial monoid $F$ such that $F^\times \cap M = M^\times$ and $\mathcal{C}^*(M,F)$ is finite.
\end{definition}

With notation as in Definition~\ref{def:C-monoid}, we say that $M$ is a C-monoid \emph{defined in} $F$. A C-monoid can be defined in more than one factorial monoid $F$; however, there is a canonical way of choosing $F$ (see \cite[Theorem~5.6.A.3]{GH06a}). Because C-monoids are submonoids of factorial monoids, they are atomic. The family of C-monoids allows us to study the arithmetic of non-integrally closed Noetherian domains.

Given a multiplicative monoid $M$ with quotient group $\mathsf{g}(M)$, we say that $x \in \mathsf{g}(M)$ is \emph{almost integral} over $M$ if there exists $c \in M$ such that $cx^n \in M$ for every $n \in \nn$. The subset of $\mathsf{g}(M)$ consisting of all almost integral elements over $M$ is denoted by $\widehat{M}$ and called the \emph{complete integral closure} of $M$.

Let $R$ be an integral domain with field of fractions $\mathsf{q}(R)$. An ideal $I$ of $R$ is \emph{divisorial} if $(R : (R : I)) = I$. The domain $R$ is called a \emph{Mori domain} if it satisfies the ascending chain condition on divisorial ideals. Finally, for the domain $R$ we set $\widehat{R} = \widehat{R^\bullet} \cup \{0\}$, where $R^\bullet$ is the multiplicative monoid of $R$.

\begin{example}
	If $A$ is a Mori domain, then $R = \widehat{A}$ is a Krull domain. Moreover, if $\mathfrak{f} = (A : R)$ is nonzero and both the quotient ring $R/\mathfrak f$ and the class group $\mathcal C (R)$ are finite, then $A^\bullet$ is a C-monoid (see \cite[Theorem 2.11.9]{GH06b}). More examples of C-monoids can be found in \cite{GRR15} and \cite{aR13}.
\end{example}

The next theorem is used in the proof of Proposition~\ref{prop:PM are almost never transfer C-monoids}.

\begin{theorem} \cite[Theorem 2.9.11(2)]{GH06b} \label{thm:complete integral closures of C-monoids are Krull monoids}
	The complete integral closure of a \emph{C}-monoid is a Krull monoid.
\end{theorem}

\begin{prop} \label{prop:PM are almost never transfer C-monoids}
	A nontrivial Puiseux monoid is a \emph{C}-monoid if and only if it is isomorphic to a numerical monoid.
\end{prop}

\begin{proof}
	Let $M$ be a nontrivial Puiseux monoid that is also a C-monoid. Let $\widehat{M}$ be the complete integral closure of $M$. Observe first that if $x \in \mathsf{g}(M) \cap \qq_{<0}$, then
	\[
		S_{x,r} := \{r + nx \mid n \in \nn\}
	\]
	contains only finitely many positive rational numbers for all $r \in M$. As a result, $|S_{x,r} \cap M| < \infty$ for all $r \in M$, which implies that no negative element of $\mathsf{g}(M)$ is almost integral over $M$. Because $\widehat{M}$ is a monoid and it is contained in $\qq_{\ge 0}$, it must be a Puiseux monoid. By Theorem~\ref{thm:complete integral closures of C-monoids are Krull monoids}, the monoid $\widehat{M}$ is a Krull monoid; in particular, it is transfer Krull. Now Theorem~\ref{thm:Puiseux monoids are not transfer Krull} ensures that $\widehat{M}$ is isomorphic to $(\nn_0,+)$. Finally, the fact that $M$ is a submonoid of $\widehat{M}$ forces $M$ to be isomorphic to a numerical monoid.
	
	For the reverse implication, it suffices to note that for every proper numerical monoid $N$, any two natural numbers greater than the Frobenius number of $N$ are $N$-equivalent, which implies that $\mathcal{C}^*(N, \nn_0)$ is finite. As $\nn_0$ is also a C-monoid, the proof follows.
\end{proof}
\medskip

\section{Acknowledgements}

	While working on this paper, I was supported by the NSF-AGEP fellowship. I am grateful to Alfred Geroldinger not only for proposing the main questions motivating this paper but also for his enlightening guidance through early drafts. Also, I would like to thank Salvatore Tringali and the anonymous referee, whose helpful suggestions help me improve the final version of this paper.

\end{document}